\documentclass[letterpaper,reqno]{amsart}

\usepackage{amsmath} 
\usepackage{amsfonts} 
\usepackage{amssymb} 
\usepackage{epsfig} 
\usepackage{graphicx} 
\usepackage{multirow} 
\usepackage{verbatim} 
\usepackage{rotating} 
\usepackage[all]{xy} 
\usepackage{appendix}
\usepackage{tikz}
\usepackage{dsfont}
\usetikzlibrary{matrix}

\newtheorem{theorem}{Theorem}[section]
\newtheorem{lemma}[theorem]{Lemma}

\theoremstyle{definition}

\newtheorem{proposition}[theorem]{Proposition}

\theoremstyle{remark}

\theoremstyle{notation}
\newtheorem{notation}[theorem]{Notation}
\numberwithin{equation}{section}
\theoremstyle{corollary}

\newtheorem{conjecture}[theorem]{Conjecture}
\usepackage[bookmarks=false]{hyperref}

\begin{document}

\title[]{A remark on the Farrell-Jones conjecture}

\author[]{Ilias Amrani}
\address{Department of Mathematics and Information Technology, Academic University of Saint Petersburg, Russian Federation.}

\address{Faculty of Mathematics and Mechanics, Saint Petersburg State University, Russian Federation.}
\email{ilias.amranifedotov@gmail.com}


\subjclass[2000]{19D50, 55P47, 55N20, 55P20, 18F25, 18E30.}


\keywords{Algebraic $K$-theory, Waldhausen $K$-theory, Farrell-Jones Conjecture, Group algebra.}

\begin{abstract}
Assuming the classical Farrell-Jones conjecture we produce an explicit (commutative) group ring $R$ and a thick subcategory $\mathsf{C}$ of perfect $R$-complexes such that the Waldhausen $K$-theory space $\mathrm{K}(\mathsf{C})$ is equivalent to a rational Eilenberg-Maclane space. 
\end{abstract}

\maketitle
\section{Introduction}
Our main goal is to prove the following theorem
\begin{theorem}[Main result]\label{maint}\ref{theorem1}
There exists a commutative ring $R$ and a thick subcategory $\mathsf{C}$ of $\mathsf{Perf}(R)$ such that the space $\mathrm{K}(\mathsf{C})$ of Waldhausen $K$-theory is equivalent to an Eilenberg-MacLane space. 
\end{theorem} 
In our opinion this theorem seems counterintuitive at the first glance. There is very few examples of rings for which the algebraic $K$-theory groups were computed in all degrees (e.g. the $K$-theory of finite fields computed by Quillen). Another source for such computation is the Farrell-Jones conjecture. We will compute explicitly the $K$-groups for some particular (commutative) group ring \ref{F-J}.
\begin{conjecture}[Classical Farrell-Jones \cite{luck2010k}]\label{1}
For any regular ring $k$ and any torsionfree group $G$, the assembly map 
$$\mathrm{H}_{n}(\mathrm{B}G; \mathbf{K}(k))\longrightarrow \mathrm{K}_{n}(k[G]) $$
is an isomorphism for any $n\in \mathds{Z}$. 
\end{conjecture}
We refer to \cite{waldhausen1985algebraic} for the definition of the $K$-theory spectrum $\mathbf{K}(k)$ of a ring $k.$ We recall that $\mathrm{B}G$ is the classifying space of the group $G$ and that $k[G]$ is the associated group ring with a natural augmentation $k[G]\rightarrow k$. 
We recall also that $\mathrm{H}_{n}(\mathrm{B}G; \mathbf{K}(k))$ is the same thing as the $n$-th stable homotopy group of the spectrum $\mathrm{B}G_{+}\wedge \mathbf{K}(k)$. More precisely the assembly map is induced by the following map of spectra  
$$\mathrm{B}G_{+}\wedge \mathbf{K}(k)\rightarrow \mathbf{K}(k[G]). $$
The conjecture \ref{1} admits a positive answer in the case where $k$ is regular ring and $G$ is a torsionfree abelian group, it is a particular case of the main result of \cite{wegner2015farrell}.

\section{Fibre sequence for Waldhausen K-theory }
\begin{notation}\label{notaion} We fix the following notations:
\begin{enumerate}
\item Let $\mathcal{E}$ be any (differential graded) ring. Let $ \mathsf{Mod}_{\mathcal{E}}$ denotes the (differential graded) model category of $\mathcal{E}$-complexes \cite{hovey2007model}. And $\mathsf{Perf}(\mathcal{E})$ denotes the (differential graded) category of perfect (i.e. compact) $\mathcal{E}$-complexes.
\item For any (differential graded) ring map $\mathcal{E}\rightarrow \mathcal{A}$, $\mathsf{Perf}(\mathcal{E},\mathcal{A})$ denotes the thick subcategory of $\mathsf{Perf}(\mathcal{E})$ such that $M\in \mathsf{Perf}(\mathcal{E},\mathcal{A})$ if and only if $M\otimes^{\mathbb{L}}_{\mathcal{E}}\mathcal{A}\simeq 0$ i.e.  $M\otimes^{\mathbb{L}}_{\mathcal{E}}\mathcal{A}$ is quasi-isomorphic to $0$. By the symbol $\otimes^{\mathbb{L}}_{\mathcal{E}}$ we do mean the derived tensor product over $\mathcal{E}$.
\end{enumerate}
\end{notation}
\begin{lemma}\label{main}
Let $\mathcal{E}\rightarrow \mathcal{A}$ be a morphism of (differential graded) rings such that $\mathcal{A}\otimes^{\mathbb{L}}_{\mathcal{E}}\mathcal{A}\simeq \mathcal{A}$, then
$$ \mathrm{K}(\mathcal{E},\mathcal{A})\rightarrow \mathrm{K}(\mathcal{E})\rightarrow  \mathrm{K}(\mathcal{A})$$
is a fibre sequence of (infinite loop) spaces where $\mathrm{K}(\mathcal{E},\mathcal{A}):= \mathrm{K}(\mathsf{Perf}(\mathcal{E},\mathcal{A}))$.
\end{lemma}
\begin{proof}
Let $\mathbf{w}$ be the class of equivalences in $ \mathsf{Mod}_{\mathcal{E}}$ defined as follows:  a map $P\rightarrow P^{'} $ is $\mathbf{w}$-equivalence if and only if $\mathcal{A}\otimes^{\mathbb{L}}_{\mathcal{E}}P\rightarrow \mathcal{A}\otimes^{\mathbb{L}}_{\mathcal{E}}P^{'} $ is a quasi-isomorphism ($\mathbf{q.i.}$).

The left Bousfield localization \cite{hirschhorn2009model} of the model category $\mathsf{Mod}_{\mathcal{E}}$ with respect to the class $\mathbf{w}$ exists and it is denoted by $\mathrm{L}_{\mathbf{w}} \mathsf{Mod}_{\mathcal{E}}$. Since $\mathcal{A}\otimes^{\mathbb{L}}_{\mathcal{E}}\mathcal{A}\simeq \mathcal{A}$ we obtain a Quillen equivalence 
$$\xymatrix{\mathrm{L}_{\mathbf{w}}\mathsf{Mod}_{\mathcal{E}}\ar@<2pt>[r]^-{\mathcal{A}\otimes_{\mathcal{E}}-} & \mathsf{Mod}_{\mathcal{A}}\ar@<2pt>[l]^-{U}} $$
More precisely, for any $M\in \mathsf{Mod}_{\mathcal{A}}$ the (derived) counit map $ \mathcal{A}\otimes^{\mathbb{L}}_{\mathcal{E}} U(M)\rightarrow M $
is a quasi-isomorphism (because it is a quasi-isomorphism for $\mathcal{A}=M$, the functor $\mathcal{A}\otimes_{\mathcal{E}}^{\mathbb{L}}-$ commutes with homotopy colimits and $\mathcal{A}$ is a generator for the homotopy category of $\mathsf{Mod}_{\mathcal{A}}$). On another hand, the derived unit map $P\rightarrow \mathcal{A}\otimes^{\mathbb{L}}_{\mathcal{E}} U(P)$ is an equivalence in $\mathrm{L}_{\mathbf{w}}\mathsf{Mod}_{\mathcal{E}}$ for any $P\in \mathsf{Mod}_{\mathcal{E}}$ by definition. In particular the subcategory of compact objects in $\mathrm{L}_{\mathbf{w}}\mathsf{Mod}_{\mathcal{E}}$ is equivalent to $\mathsf{Perf}(\mathcal{A}$). Thus, by \cite[theorem 3.3]{sagave2004algebraic}, we have an equivalence of the $K$-theory spaces
$$\mathrm{K}((\mathsf{Perf}(\mathcal{E}),\mathbf{w}))\simeq \mathrm{K}((\mathsf{Perf}(\mathcal{A}),\mathbf{q.i.})):= \mathrm{K}(\mathcal{A}).$$ 
By Waldhausen fundamental theorem \cite[Theorem 1.6.4]{waldhausen1985algebraic}, the sequence of Waldhausen categories   
$$ (\mathsf{Perf}(\mathcal{E})^{\mathbf{w}},\mathbf{q.i.})\rightarrow (\mathsf{Perf}(\mathcal{E}),\mathbf{q.i.})\rightarrow (\mathsf{Perf}(\mathcal{E}),\mathbf{w}) $$
induces a fibre sequence of $K$-theory spaces
$$\mathrm{K}((\mathsf{Perf}(\mathcal{E})^{\mathbf{w}},\mathbf{q.i.}))\rightarrow \mathrm{K}(\mathcal{E})\rightarrow  \mathrm{K} (\mathcal{A})$$
 where $\mathsf{Perf}(\mathcal{E})^{\mathbf{w}}$ is the full subcategory of $\mathsf{Perf}(\mathcal{E})$ such that $E\in \mathsf{Perf}(\mathcal{E})^{\mathbf{w}}$ if and only if $\mathcal{A}\otimes^{\mathbb{L}}_{\mathcal{E}}E\simeq 0$. It is obvious by definition that  $\mathsf{Perf}(\mathcal{E})^{\mathbf{w}}=\mathsf{Perf}(\mathcal{E},\mathcal{A})$.\\
Hence $$\mathrm{K}(\mathcal{E},\mathcal{A})\rightarrow \mathrm{K}(\mathcal{E})\rightarrow \mathrm{K}(\mathcal{A})$$ is a homotopy fibre sequence of spaces.
\end{proof}
A similar result can be found in \cite[Theorem 0.5]{neeman2004noncommutative} and in \cite[Lemma 5.1]{chen2012recollements}.

\section{Farrell-Jones conjecture}

\begin{notation} We fix the following notations:
\begin{enumerate}
\item $k=\mathds{F}_{2}$ is the finite field with two elements.
\item $R$ is the group algebra $k[\mathds{Q}]$, where $\mathds{Q}$ is the additive abelian group of rational numbers. 
\end{enumerate}
\end{notation}
\begin{proposition}\label{prop}
If $\mathds{V}$ is a rational vector space and $A$ is a finite abelian group then 
\[\mathrm{H}_{\ast}(\mathrm{B}\mathds{V}; \mathds{Z})=
\begin{cases}
\mathds{Z}& \quad \textrm{ if } n= 0\\
\mathds{V} & \quad \textrm{ if } n= 1\\
0 & \quad \textrm{ else } \\
\end{cases}\]
and 
\[\mathrm{H}_{\ast}(\mathrm{B}\mathds{V}; A)=
\begin{cases}
A & \quad \textrm{ if } n= 0\\
0 & \quad \textrm{ else } \\
\end{cases}\]

\end{proposition}
\begin{lemma}\label{F-J}
 
\[\pi_{n}\mathrm{K}(R):=\mathrm{K}_{n}(R)=
\begin{cases}
\mathrm{K}_{n}(k) & \quad \textrm{ if } n\neq 1\\
 \mathds{Q} & \quad \textrm{ if } n=1 \\
\end{cases}\]
\end{lemma}
\begin{proof}
By Quillen theorem \cite{quillen1972cohomology}, the algebraic $K$-theory of the finite field $k$ is given by
\[\mathrm{K}_{n}(k)=
\begin{cases}
\mathds{Z} & \quad \text{ if } n=0 \\
0 & \quad \text{ if } n \text{ even } >0 \\
 \mathds{Z}/(2^{j}-1) & \quad \text{ if } n=2j-1 \text{ and } j>0 \\
\end{cases}
\]
Since $\mathds{Q}$ is a rational vector space and $\mathrm{K}_{n}(k)$ are finite abelian groups (for $n>0$) then by proposition \ref{prop} we have that
\[\mathrm{H}_{p}(\mathrm{B}\mathds{Q}; \mathrm{K}_{q}(k))=
\begin{cases}
\mathds{Q} & \quad \text{ if } p=1 \text{ and } q=0\\
\mathrm{K}_{q}(k) & \quad \text{ if } p=0 \text{ and } q\geq 0 \\
0 & \quad \text{ else }\\
\end{cases}
\]
The second page $E^{2}_{p,q}=\mathrm{H}_{p}(\mathrm{B}\mathds{Q}; \mathrm{K}_{q}(k))$ of the converging Atiyah-Hirzebruch spectral sequence \cite{luck2010k} 
$$\mathrm{H}_{p}(\mathrm{B}\mathds{Q}; \mathrm{K}_{q}(k))  \Longrightarrow\mathrm{H}_{p+q}(\mathrm{B}\mathds{Q}; \mathbf{K}(k))$$
has graphically the following shape:\\
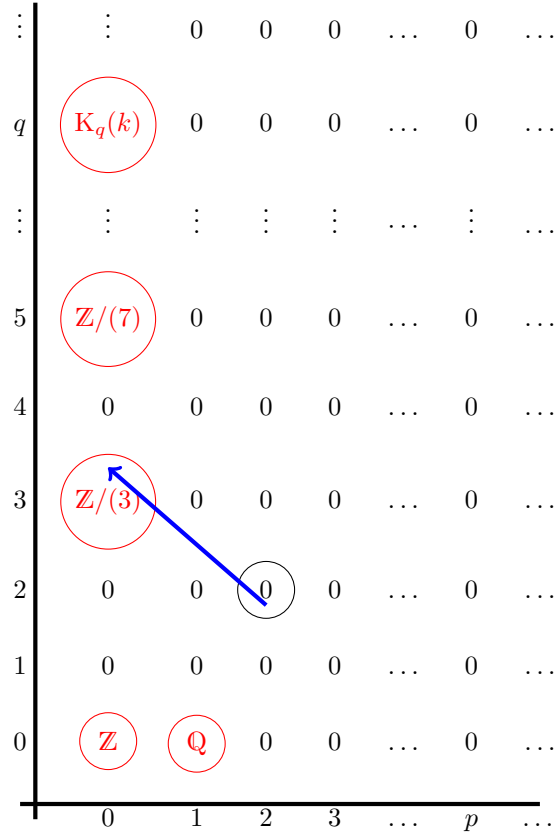
\begin{center}
\begin{tikzpicture}
  \matrix (m) [matrix of math nodes,
    nodes in empty cells,nodes={minimum width=5ex,
    minimum height=5ex,outer sep=-5pt},
    column sep=1ex,row sep=1ex]{
    \vdots     &  \vdots  &      0        &   0   &   0         &\dots &  0  &\dots\\
    q     &  |[draw,red,circle]|\mathrm{K}_{q}(k)  & 0             &   0   &   0         &\dots &  0  &\dots\\
          \vdots     &  \vdots  & \vdots             & \vdots  & \vdots              &  \cdots & \vdots &\dots  \\
          5     &  |[draw,red,circle]|\mathds{Z}/(7) &  0          &   0 &   0   & \dots & 0  &\dots\\
          4     &  0 &  0          &   0 &   0   & \dots & 0  &\dots\\
          3     &  |[draw,red,circle]|\mathds{Z}/(3) &  0          &   0 &   0   & \dots & 0  &\dots\\
          2     &  0 &  0          &  |[draw,circle]|0  &   0   & \dots &0  &\dots\\
          1     &  0 &  0         &   0   &   0 & \dots &0  &\dots\\
          0     & |[draw,red,circle]| \mathds{Z} & |[draw,red,circle]|\mathds{Q} &  0   &   0 & \dots &0  &\dots\\
    \quad\strut &   0  &  1  &  2  & 3 &    \dots & p &\dots\strut\\ };
\draw[ultra thick] (m-1-1.north east) -- (m-10-1.east) ;
\draw[ultra thick] (m-10-1.north) -- (m-10-8.north) ;
\draw[blue,ultra thick,->] (m-7-4.south) -- (m-6-2.north);
\end{tikzpicture}\\
\end{center}
where the differentials $d^{2}: E_{p,q}^{2}\rightarrow E_{p-2,q+1}^{2}$ are obviously identical to $0$. It means that the spectral sequence collapses, hence in our particular case it implies that
$$ \mathrm{H}_{p}(\mathrm{B}\mathds{Q}; \mathrm{K}_{q}(k))= \mathrm{H}_{p+q}(\mathrm{B}\mathds{Q}; \mathbf{K}(k)).$$
Since the Farrell-Jones conjecture is true in the case of torsionfree abelian groups \cite{wegner2015farrell}, we obtain that 
\[\mathrm{K}_{n}(R)\cong\mathrm{H}_{n}(\mathrm{B}\mathds{Q}; \mathbf{K}(k))=
\begin{cases}
    \mathrm{K}_{n}(k)  & \quad \text{if } n\neq 1 \\
    \mathds{Q} & \quad \text{ if } n=1\\
   \end{cases}
\]
\end{proof}

\begin{lemma}\label{lemma}
There is a fibre sequence of Waldhausen $K$-theory spaces given by 
$$ \mathrm{K}(R,k)\rightarrow \mathrm{K}(R)\rightarrow  \mathrm{K}(k)$$
\end{lemma}
\begin{proof}
Since $k$ is a finite field (in particular a finite abelian group) and $\mathds{Q}$ is a rational vector space, it follows by \ref{prop} that 

\[\mathrm{H}_{n}(\mathrm{B}\mathds{Q};k)=\mathrm{Tor}^{R}_{n}(k,k)=\begin{cases}
    k & \quad \text{if } n=0 \\
    0 & \quad \text{ else } \\
   \end{cases}
\]
therefore $k\otimes_{R}^{\mathbb{L}}k\simeq k$. The conclusion follows from lemma \ref{main} when $k=\mathcal{A}$ and $R=\mathcal{E}$. 
\end{proof}

\begin{theorem}\label{theorem1}
With the same notation, the $K$-theory space of the thick subcategory $\mathsf{Perf}(R,k)$ is equivalent to the Eilenberg-MacLane space $\mathrm{B}\mathds{Q}$.
\end{theorem}
\begin{proof}

Since the Farrell-Jones conjecture is true for $G=\mathds{Q}$. Combining lemma \ref{lemma} and lemma \ref{F-J}, we have by Serre's long exact sequence that the homotopy groups of the homotopy fibre $\mathrm{K}(R,k)$ of 
$\mathrm{K}(R)\rightarrow \mathrm{K}(k)$ are given by 
 \[\mathrm{K}_{n}(R,k)=\begin{cases}
    \mathds{Q} & \quad \text{if } n=1 \\
    0 & \quad \text{ else } \\
   \end{cases}
\]
and by definition $\mathrm{K}(R,k):=\mathrm{K}(\mathsf{Perf}(R,k))$, hence we have proved the main theorem \ref{maint}. 
\end{proof}

 \bibliographystyle{plain}
\bibliography{FJEMK}

\begin{thebibliography}{1}

\bibitem{chen2012recollements}
Hongxing Chen and Changchang Xi.
\newblock Recollements of derived categories {II}: Algebraic {K}-theory.
\newblock {\em arXiv:1212.1879}, 2012.

\bibitem{hirschhorn2009model}
Philip~S. Hirschhorn.
\newblock {\em Model categories and their localizations}.
\newblock Number~99. American Mathematical Soc., 2009.

\bibitem{hovey2007model}
Mark Hovey.
\newblock {\em Model categories}.
\newblock Number~63. American Mathematical Soc., 2007.

\bibitem{luck2010k}
Wolfgang L{\"u}ck.
\newblock {K}-and {L}-theory of group rings.
\newblock {\em arXiv:1003.5002}, 2010.

\bibitem{neeman2004noncommutative}
Amnon Neeman and Andrew Ranicki.
\newblock Noncommutative localisation in algebraic {K}--theory {I}.
\newblock {\em Geometry \& Topology}, 8(3):1385--1425, 2004.

\bibitem{quillen1972cohomology}
Daniel Quillen.
\newblock On the cohomology and {K}-theory of the general linear groups over a
  finite field.
\newblock {\em Annals of Mathematics}, pages 552--586, 1972.

\bibitem{sagave2004algebraic}
Steffen Sagave.
\newblock On the algebraic {K}-theory of model categories.
\newblock {\em Journal of Pure and Applied Algebra}, 190(1):329--340, 2004.

\bibitem{waldhausen1985algebraic}
Friedhelm Waldhausen.
\newblock Algebraic {K}-theory of spaces.
\newblock In {\em Algebraic and {G}eometric {T}opology}, pages 318--419.
  Springer, 1985.

\bibitem{wegner2015farrell}
Christian Wegner.
\newblock The {F}arrell--{J}ones conjecture for virtually solvable groups.
\newblock {\em Journal of Topology}, 8(4):975--1016, 2015.

\end{thebibliography}
\end{document}